\documentclass[11pt,a4paper]{amsart}
\usepackage{amsmath,amsthm,amssymb,amsfonts}
\usepackage{amsaddr}
\usepackage{graphicx}
\usepackage{epstopdf}
\usepackage{color}
\usepackage{hyperref}
\usepackage[shortlabels]{enumitem}
\newtheorem{theorem}{Theorem}[section] 
\newtheorem{proposition}[theorem]{Proposition} 
\newtheorem{lemma}[theorem]{Lemma} 
\newtheorem{corollary}[theorem]{Corollary}

\newtheorem{construction}[theorem]{Construction}

\theoremstyle{remark}

\newtheorem{remark}[theorem]{Remark}
\newtheorem{remarks}[theorem]{Remarks}

\theoremstyle{definition}
\newtheorem{definition}[theorem]{Definition} 

\def\cD{\mathcal{D}}

\def\cP{\mathcal{P}}

\def\cS{\mathcal{S}}
\def\cT{\mathcal{T}}

\def\RR{\mathbb{R}}

\def\RR{\mathbb{R}}

\DeclareMathOperator{\conv}{conv}

\DeclareMathOperator{\verts}{vert}

\DeclareMathOperator{\relint}{relint}

\DeclareMathOperator{\signe}{sign}
\newcommand{\sprod}[2]{\langle {#1} , {#2} \rangle} %
\newcommand{\defn}[1]{\emph{#1}} 
\newcommand{\set}[2]{\ensuremath{\left\{#1\,\middle|\,#2\right\}}} %
\newcommand{\ffloor}[2]{\left\lfloor{\frac{#1}{#2}}\right\rfloor} %

\newcommand{\wh}[1]{\widehat{#1}}
\newlength{\dhatheight}
\newcommand{\wwh}[1]{%
\settoheight{\dhatheight}{\ensuremath{\widehat{#1}}}%
\addtolength{\dhatheight}{-0.4ex}%
\widehat{\smash{\widehat{#1}}\phantom{\rule{3pt}{\dhatheight}}}}
\def\e{\mathrm{e}}
\newcommand{\cyc}[2]{C_{#1}(#2)}

\newcommand{\inpK}[2]{\operatorname{ins}_K({#1,#2})}

\newcommand{\Del}{\cD}
\newcommand{\sscb}{ssc} %

\newcommand{\np}{c}%
\newcommand{\sproj}{s}%
\newcommand{\csproj}[1]{\sproj_{#1}}%
\newcommand{\csp}{\csproj{\np}}%
\newcommand{\isp}{\csp^{-1}}%

\newcommand{\lex}{lex} %

\newcommand{\ball}[1]{B_{#1}} %

\begin{document}

\title{Neighborly inscribed polytopes and Delaunay triangulations}

\author{Bernd Gonska and Arnau Padrol}
\email{GonskaB@gmail.com}

\thanks{The research of Arnau Padrol is supported by the DFG Collaborative Research Center SFB/TR~109 ``Discretization in Geometry and Dynamics''}
\address{Institut f\" ur Mathematik, Freie Universit\"at Berlin\\
 Arnimallee 2, 14195 Berlin, Germany}
\email{arnau.padrol@fu-berlin.de}

\begin{abstract}

We construct a large family of neighborly polytopes that can be realized 
with all the vertices on the boundary of any smooth strictly convex body.
In particular, we show that there are superexponentially many combinatorially 
distinct neighborly polytopes that admit realizations inscribed on the sphere.
These are the first examples of inscribable neighborly polytopes that
are not cyclic polytopes, and provide the current best lower bound for the
number of combinatorial types of inscribable polytopes (which coincides with the 
current best lower bound for the number of combinatorial types of polytopes).
Via stereographic projections, this translates into a superexponential lower 
bound for the number of combinatorial types of (neighborly) Delaunay triangulations.


\end{abstract}
\maketitle

\section{Introduction}
\label{intro}

A polytope is called \defn{inscribable} if it has a realization with all its
vertices on the sphere. 
The question of whether every polytope is inscribable was first considered by
Steiner in 1832~\cite{Steiner1832} and the first negative examples were found by
Steinitz in 1928~\cite{Steinitz1928}. 
For example, the polytope obtained by stacking a tetrahedron on top of each
facet of a tetrahedron is not inscribable (on the sphere). 

Via stereographic
projections, inscribable polytopes are in 
correspondence with Delaunay triangulations \cite{Brown1979}, 
which are central objects in computational 
geometry~\cite{Edelsbrunner2006}. Their applications include nearest-neighbors search, 
pattern matching, clustering and mesh generation, among others. This has triggered a 
renewed interest in inscribable polytopes.

Deciding whether a polytope is inscribable is in general a hard 
problem~\cite{AdiprasitoPadrolTheran2014,PadrolTheran2014}.
In dimension~$3$ this can be done efficiently thanks to a fundamental characterization 
found by Hodgson, Rivin and Smith using angle structures and hyperbolic geometry 
\cite{HodgsonRivinSmith1992,Rivin1996}. In higher dimensions, however, the 
question of inscribability is still wide open.
It is not even known which vectors appear as $f$-vectors of inscribable simplicial polytopes~\cite{Gonskaphd}.

By McMullen's Upper Bound Theorem~\cite{McMullen1970}, the complexity
of inscribed polytopes (and Delaunay triangulations) is bounded by that of 
neighborly polytopes.
The existence of inscribed neighborly polytopes
was already known to Carath\'eodory in 1911, when he presented an inscribed
realization of the cyclic polytope \cite{Caratheodory1911}.
Inscribed cyclic polytopes were also used by Seidel~\cite{Seidel1987,Seidel1991} in 
the context of an upper bound theorem for Delaunay triangulations.
While more inscribed realizations of the cyclic polytope are known
(c.f.\ \cite{Gonskaphd}, \cite{GonskaZiegler2013}), 
no other example of inscribable neighborly polytope has been found.
The centrally symmetric analogues of cyclic polytopes constructed by Barvinok and Novik
are also inscribed~\cite{BarvinokNovik2008}.

Without the constraint of inscribability, Gr\"unbaum
found the first examples of non-cyclic neighborly
polytopes~\cite{GrunbaumConvexPolytopes}. Even more, Shemer used the \emph{sewing construction} to prove in 1982
that the number of combinatorial types of neighborly $d$-polytopes with $n$
vertices is of order $n^{\frac n2(1+o(1))}$~\cite{Shemer1982}. This bound was
recently improved in \cite{Padrol2013, Padrolphd}, by proposing a new construction for neighborly
polytopes that contains Shemer's family. 
Even if this method cannot generate all neighborly polytopes, it can be used to
show that there are at least
$n^{\frac{d}{2}n(1+o(1))}$ different combinatorial types of labeled neighborly
$d$-polytopes with $n$ vertices (as $n\rightarrow\infty$ with $d$ fixed). This
is currently also the best lower bound for the number of combinatorial types of
labeled $d$-polytopes with $n$ vertices.

Our main contribution in this paper is to show that all these neighborly
polytopes are inscribable.
To this end, we revisit the construction in \cite{Padrol2013,Padrolphd} using a
technique developed in~\cite{Gonskaphd,GonskaZiegler2013} to construct
inscribable cyclic polytopes via Delaunay triangulations. 
This provides a very simple construction
(Theorem~\ref{thm:inscribableandneighborly} and
Construction~\ref{cons:construction}) for high dimensional inscribable
neighborly polytopes (and hence also for neighborly Delaunay triangulations  and
dual-to-neighborly Voronoi diagrams). 
With it we conclude (see
Theorem~\ref{thm:bound}) that the number of different labeled
combinatorial types of inscribable neighborly $d$-polytopes with $n$
vertices is at least $n^{\frac{d}{2}n(1+o(1))}$ (as
$n\rightarrow \infty$ with $d$ fixed). As a reference, the best upper bound
for the number of different labeled combinatorial types of $d$-polytopes 
with $n$ vertices is of order $\left({n}/{d} \right)^{d^2n(1+o(1))}$ when $\tfrac{n}{d}\rightarrow
\infty$ (see~\cite{Alon1986} and~\cite{GoodmanPollack1986}). 

Actually, we prove a stronger result, since we see that all the polytopes in
this large family are $K$-inscribable for any smooth strictly convex body $K$.
That is, that they admit a realizaton with all the vertices on the boundary of
$K$. The existence of arbitrarily large families of \emph{universally inscribable} polytopes
(universal referring to all smooth strictly convex bodies) is a new result, 
to the best of the authors' knowledge (although there are alternative
methods to construct them, see Section~\ref{sec:universallyinscribable}). 
We complement this result by constructing a universally inscribable 
stacked polytope, which shows that the Lower Bound Theorem is also 
attained for simplicial polytopes in this family.
A related result of Schramm states that every $3$-polytope admits a realization
with all the edges tangent to any smooth strictly convex body~\cite{Schramm1992}.
In the converse direction, Ivanov proved that there exist 
\emph{universally circumscribing} convex bodies $K\subset \RR^d$ that fulfill 
that every $d$-polytope is $K$-inscribable~\cite{mathoverflow}. 

This begs the question of which other polytopes are universally inscribable. Also, are all (even-dimensional) 
neighborly polytopes inscribable? So far we do not
know any counterexample. Moritz Firsching found inscribed realizations for all neighborly $4$-polytopes
with up to $11$ vertices, including those that are not constructible with our methods (personal communication).

\section{Preliminaries}

If a point configuration $A=\{a_1,\dots,a_n\}\subset \RR^d$, labeled by $\{1,\dots,n\}$,
is the vertex set of a $d$-dimensional convex polytope $P$ (a \defn{$d$-polytope} from now on), we say that it is in \defn{convex position}.
Each face $F$ of $P$ can then be identified with the set of labels of the points
$a_i\in F$. The \defn{face lattice} of $P$, which is the poset of faces of $P$ 
ordered by inclusion, can thus be seen as a poset of subsets of
$\{1,\dots,n\}$. In this context, two vertex-labeled polytopes are
\defn{combinatorially equivalent} if their face lattices coincide. The
equivalence classes under this relation are called \defn{labeled combinatorial
types}. 

If $A\subset \RR^d$ is in \defn{general position}, i.e.\ no $d+1$ points of $A$ lie in
a common hyperplane, then $P$ is \defn{simplicial}, i.e.\ every facet of
$P$ is a simplex. A polytope $P$ is \defn{$k$-neighborly} if every subset of $k$
vertices
of~$P$ forms a face of~$P$. No $d$-polytope other than the simplex
can be $k$-neighborly for any $k>\ffloor{d}{2}$, which motivates the definition
of \defn{neighborly} polytopes as those that are $\ffloor{d}{2}$-neighborly.

A canonical example of neighborly polytope is the \defn{cyclic polytope},
$\cyc{n}{d}$, obtained as the convex hull of $n$ points on any $d$-order curve
in $\RR^d$ \cite{Sturmfels1987}. For example, the moment curve $\gamma:t\mapsto
(t,t^2,\dots,t^d)$ is a $d$-order curve. In even dimensions, the trigonometric
moment curve $$\tau:t\mapsto
(\sin(t),\cos(t),\sin(2t),\cos(2t),\dots,\sin(\tfrac d2t),\cos(\tfrac d 2t))$$
is a $d$-order curve on the sphere, providing inscribed realizations of the
cyclic polytope~\cite{Caratheodory1911} (see also \cite[Exercise~4.8.23]{GrunbaumConvexPolytopes}).

A \defn{triangulation} of a point configuration $A$ is a collection $\cT$ of
simplices with vertices in $A$, which we call \defn{cells}, that cover the
convex hull of $A$ and such that any pair of simplices of $\cT$ intersect in a
common face. Two triangulations of $A$ are \defn{combinatorially
equivalent} if they have the same poset of cells (as subsets of
labels of $A$).
We say that a triangulation $\cT$ of a point configuration~$A\subset\RR^d$ is
\defn{neighborly} if $\conv (S)$ is a cell of $\cT$ for each subset $S\subset A$
of size $|S|=\ffloor{d+1}{2}$.

\section{Stereographic projections}

A \defn{convex body} $K$ is a full-dimensional compact convex subset of
$\RR^d$. It is \defn{strictly convex} if its boundary $\partial K$ does not
contain any segment. A point $c\in \partial K$ in the boundary of $K$ is \defn{smooth} if it 
has a unique supporting hyperplane, and $K$ is called \defn{smooth} if every point in its boundary
is smooth.  
We abbreviate smooth
strictly convex body by \defn{\sscb-body}.

Let $K$ be an \sscb-body, let $\np$ be a point in $\partial K$ and let $H_\np$ be a translate of a
supporting hyperplane of $K$ at $\np$ that does not contain $\np$. Then for 
every $x\in \partial K\setminus \np$ the line spanned by $\np$ and $x$ 
intersects $H_\np$ in a unique point $\csp(x)$. Observe that $\csp: \partial K\setminus \np \to H_\np$
is a bijection: it is well-defined and injective because $K$ is strictly convex and surjective because $K$ is smooth.

\begin{definition}
For an \sscb-body, a point $c\in\partial K$ and a hyperplane $H_\np$ as above, the map
$\csp:\partial K\setminus \np\rightarrow \RR^{d-1}$ that maps $x$ to
$\csp(x)$ is the \defn{$K$-stereographic projection} of $\partial K$ onto $\RR^{d-1}$ with
center $\np$.
\end{definition}

Observe that the choice of the translation of the hyperplane only
changes the $K$-stereographic projection by a dilation. We
will choose $H_\np$ canonically to be tangent to $K$. 

Let $K$ be a $(d+1)$-dimensional \sscb-body and fix a $K$-stereographic
projection~$\csp$. A \defn{$K$-sphere} is the image $\csp(H\cap \partial K)$
for some hyperplane $H\subset\RR^{d+1}$ that does not
contain~$\np$. It is the boundary of the
\defn{$K$-ball} $\csp(H\cap K)\subset \RR^d$, which is an \sscb-body projectively equivalent to 
$H\cap K$. Let $a_1,\dots,a_d$ be $d$ affinely independent points
in $\RR^d$. Then the points
$\isp(a_1),\dots,\isp(a_d)\in\RR^{d+1}$ are affinely
independent and span a hyperplane $H\subset\RR^{d+1}$ that does not
contain~$\np$. We call the $K$-sphere $\csp(H\cap \partial K)$ and the $K$-ball $\csp(H\cap K)$ the
\defn{$K$-circumsphere} and the \defn{$K$-circumball} spanned by $a_1,\dots,a_d$, respectively. An example
is sketched in Figure~\ref{fig:eggstereographicprojection}.

\begin{figure}[htpb]
\includegraphics[width=.7\linewidth]{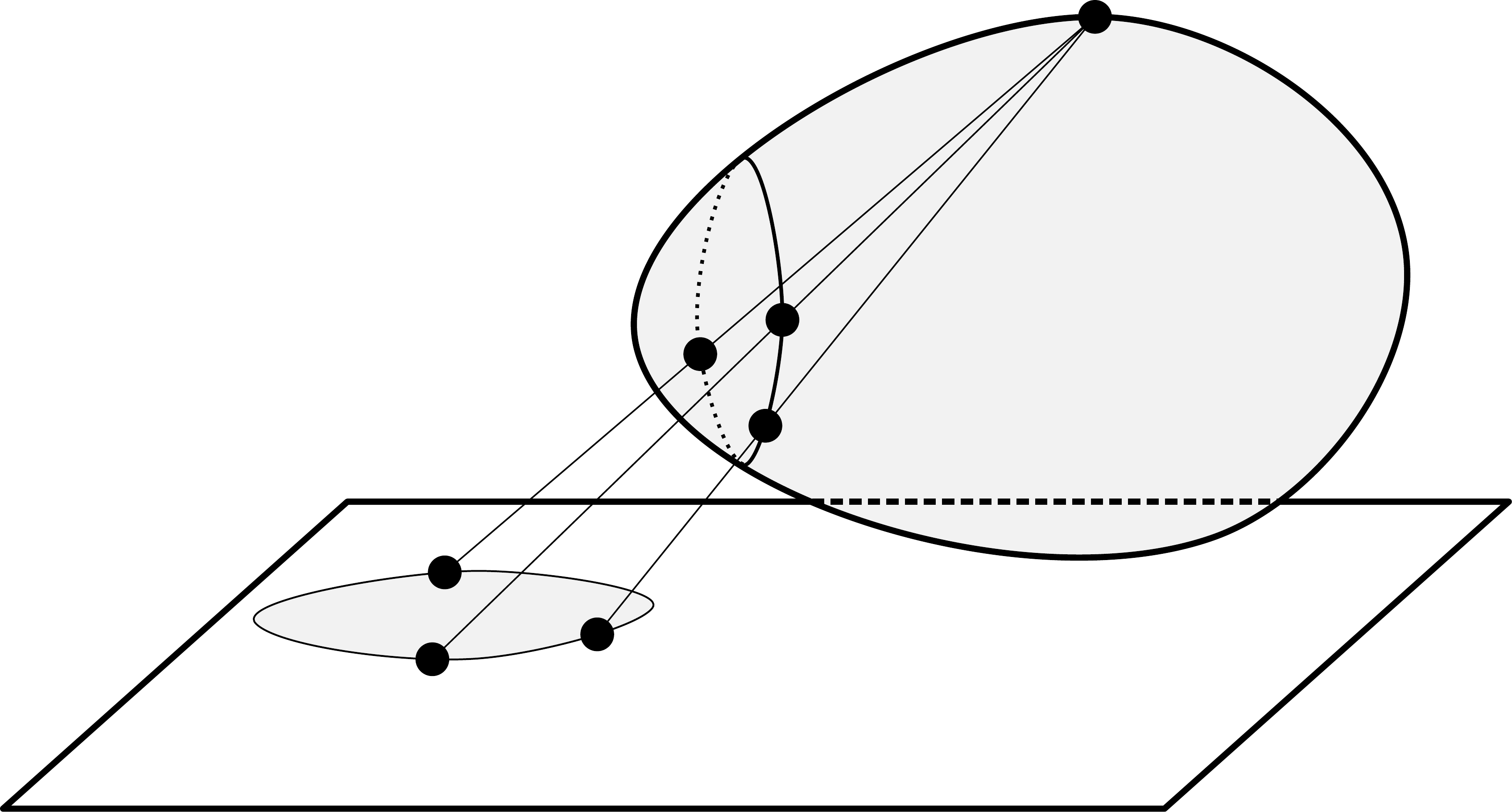}
\caption{A $K$-stereographic projection and a
$K$-circumsphere.}\label{fig:eggstereographicprojection}
\end{figure}

The key lemma concerning $K$-spheres and $K$-stereographic projections is
the following.

\begin{lemma}\label{lem:Kprojectionmatroid}
Let $K\subset \RR^{d+1}$ be an \sscb-body and let $\csp$ be a $K$-stereographic
projection. Let $a_1,\dots,a_d \in \RR^d$ be affinely independent points and let
$H$ be the hyperplane spanned by $\isp(a_1),\dots,\isp(a_d)$, oriented so that
$\np$ lies in the negative open halfspace $\np\in H^-$. 

Then a point $p\in\RR^d$ is an interior, boundary or exterior point of the
\defn{$K$-circumsphere} spanned by $a_1,\dots,a_d$ if and only if
$\isp(p)$ lies respectively in $H^+$, $H$ or $H^-$.
\end{lemma}

The \defn{$K$-Delaunay subdivision} $\Del_K(A)$ of a point configuration
$A\subset\RR^d$ consists of
all cells defined by the \defn{empty $K$-sphere condition}: for
$S\subseteq A$, $\conv(S)$ is a cell in $\Del_K(A)$ if and only if there is a
$K$-sphere that contains $S$ and has all the remaining points of
$A$ outside the corresponding $K$-ball. A cell that fulfills the empty $K$-sphere condition
is called a \defn{$K$-Delaunay cell}. These cells always form a (regular) subdivision of $A$, 
the \defn{$K$-Delaunay subdivision}, as a consequence of the upcoming Lemma~\ref{lem:KdelaunayKinscribed}.

If $A$ is in sufficiently general position  (for example if no $d+1$ points lie in
 a common hyperplane and no $d+2$ points lie on a common $K$-sphere), then the
empty $K$-sphere  condition always defines a (regular) triangulation
of~$A$, the \defn{$K$-Delaunay triangulation}.

Although the definition of $K$-spheres and $K$-balls depends on the choice of the 
$K$-stereographic projection~$\csp$, we will often simplify our statements by 
omitting the specification of~$\csp$. Hence, we will 
talk about $K$-Delaunay triangulations considering that there is some 
$K$-stereographic projection fixed beforehand.

\begin{definition}
 Let $K$ be a convex body. A polytope is \defn{$K$-inscribed} if all its
vertices lie on $\partial K$ and it is \defn{$K$-inscribable} if it is 
combinatorially equivalent to a $K$-inscribed polytope. 
\end{definition}

The following lemma ties together all the concepts presented in this section.
It is a direct consequence of Lemma~\ref{lem:Kprojectionmatroid}.

\begin{lemma}\label{lem:KdelaunayKinscribed}
Let $K$ be an \sscb-body, $\csp$ a $K$-stereographic projection, $A=\{a_1,\dots,a_{n}\}\in\RR^d$ a point configuration in $\RR^d$ and $P_A$ the $K$-inscribed polytope $P_A=\conv(\{\np\}\cup
\set{\isp(a_i)}{a_i\in A})$. 

Then for every subset $S\subseteq A$:
 \begin{enumerate}[(i)]
  \item\label{it:lowDel} $\conv(\isp({S}))$ is a face of $P_A$ if and only if
$S$ is a cell of $\Del_K(A)$, and
  \item\label{it:uppDel} $\conv(\isp({S})\cup \{\np\})$ is a face of $P_A$ if
and
only if $\conv(S)$ is a face of $\conv(A)$.
\end{enumerate}
\end{lemma}

\begin{remark}
 When we take $K$ to be $\ball{d+1}$, the $(d+1)$-dimensional unit ball, then 
 we recover the standard definitions of stereographic projection, Delaunay 
 triangulation and inscribed polytopes. In this case, though, the symmetry of 
 $\ball{d+1}$ makes these definitions independent of the choice of stereographic projection.
 Lemma~\ref{lem:KdelaunayKinscribed} for $K=\ball{d+1}$ is a classical result of
Brown~\cite{Brown1979} (cf. \cite[Proposition 0.3.13]{Gonskaphd}).
\end{remark}

\section{Liftings and triangulations}

\subsection{Lexicographic liftings}
The main tool for our construction are lexicographic liftings (which we abbreviate as \lex-liftings), which are a way to derive $(d+1)$-dimensional point configurations from $d$-dimensional point configurations. We use a variation that allows for controlling the $K$-Delaunay triangulation of the configurations, similar to a technique considered by Seidel for planar configurations~\cite{Seidel1985}.

Its definition requires the following notation. For an affine hyperplane $H=\set{x\in\RR^d}{\sprod{x}{v}=c}$ presented by a normal vector $v$ whose last coordinate is positive, a point $a\in \RR^d$ is said to be \defn{above} (resp. \defn{below}) $H$ if $\sprod{a}{v}>c$ (resp. $\sprod{a}{v}<c$).

\begin{definition}\label{def:lexicographiclifting}
Let $A=\{a_1,\dots,a_n\}$ be a configuration of $n\geq d+2$
labeled points in~$\RR^d$.

We say that a configuration $\wh A=\{\wh a_1,\dots,\wh a_n\}$ of $n$ labeled
points in~$\RR^{d+1}$ is a \defn{\lex-lifting} of $A$
(with respect to the order induced by the labeling) if $\wh
a_i=(a_i,h_i)\in\RR^{d+1}$ for each $1\leq i\leq n$,  
for some collection of heights $h_i\in \RR$ that fulfill:

\begin{enumerate}[(i)]
 \item\label{it:heights} 
 for each $i\geq d+2$,  $|h_{{i}}|$ is large enough so that $\wh a_{{i}}$ is above (if $h_{{i}}>0$)
 or below (if $h_{{i}}<0$) $H$ for every hyperplane $H$ spanned
by points in $\left\{\wh a_{{1}},\dots,\wh a_{{i-1}}\right\}$.

\end{enumerate}

Let $K$ be an \sscb-body in $\RR^{d+2}$ and $\csp$ a $K$-stereographic
projection. We call a \lex-lifting $\wh A$ of $A$ a \defn{$K$-lifting}
if moreover:
\begin{enumerate}[(i),resume]
\item\label{it:circumspheres} for each $i>d+2$, $\wh a_{{i}}$ is not contained
in any of the $K$-circumballs spanned by points in
$\left\{\wh a_{{1}},\dots,\wh a_{{i-1}}\right\}$.
\end{enumerate}
If $h_i\geq0$ for every $1\leq i \leq n$, the \lex-lifting is called \defn{positive}.
\end{definition}

Again, when working with $K$-liftings we will often omit the specification of the 
$K$-stereographic projection and assume that there is one fixed beforehand.

 If $A$ is in general position, then any \lex-lifting $\wh A$ of $A$ is
also in
general position. Furthermore, if $A$ is in convex position, then so is $\wh A$.

\begin{figure}[htpb]
\centering
 \begin{tabular}{cccc}
\includegraphics[width=.22\linewidth]{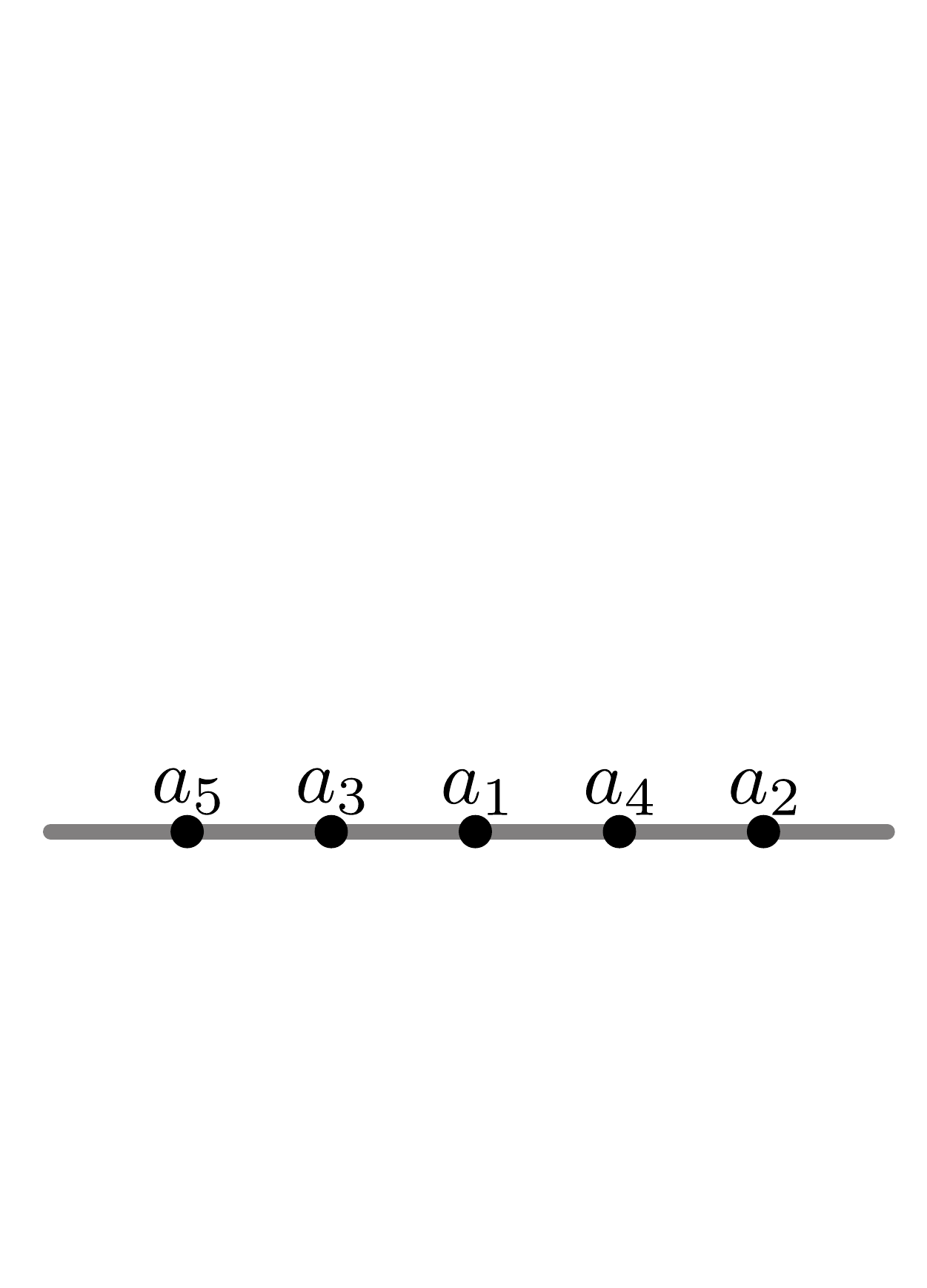}&
 \includegraphics[width=.22\linewidth]{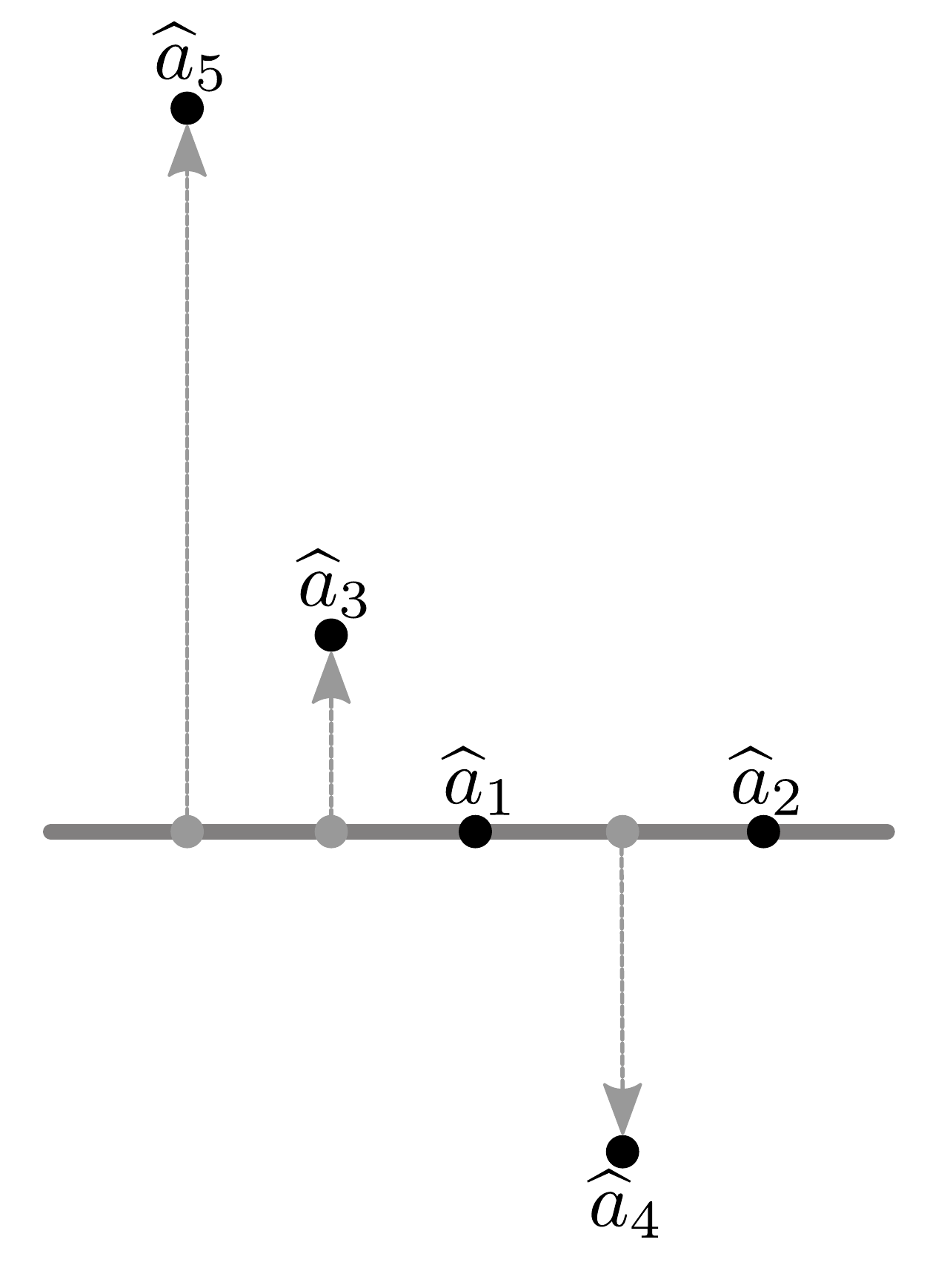}&
 \includegraphics[width=.22\linewidth]{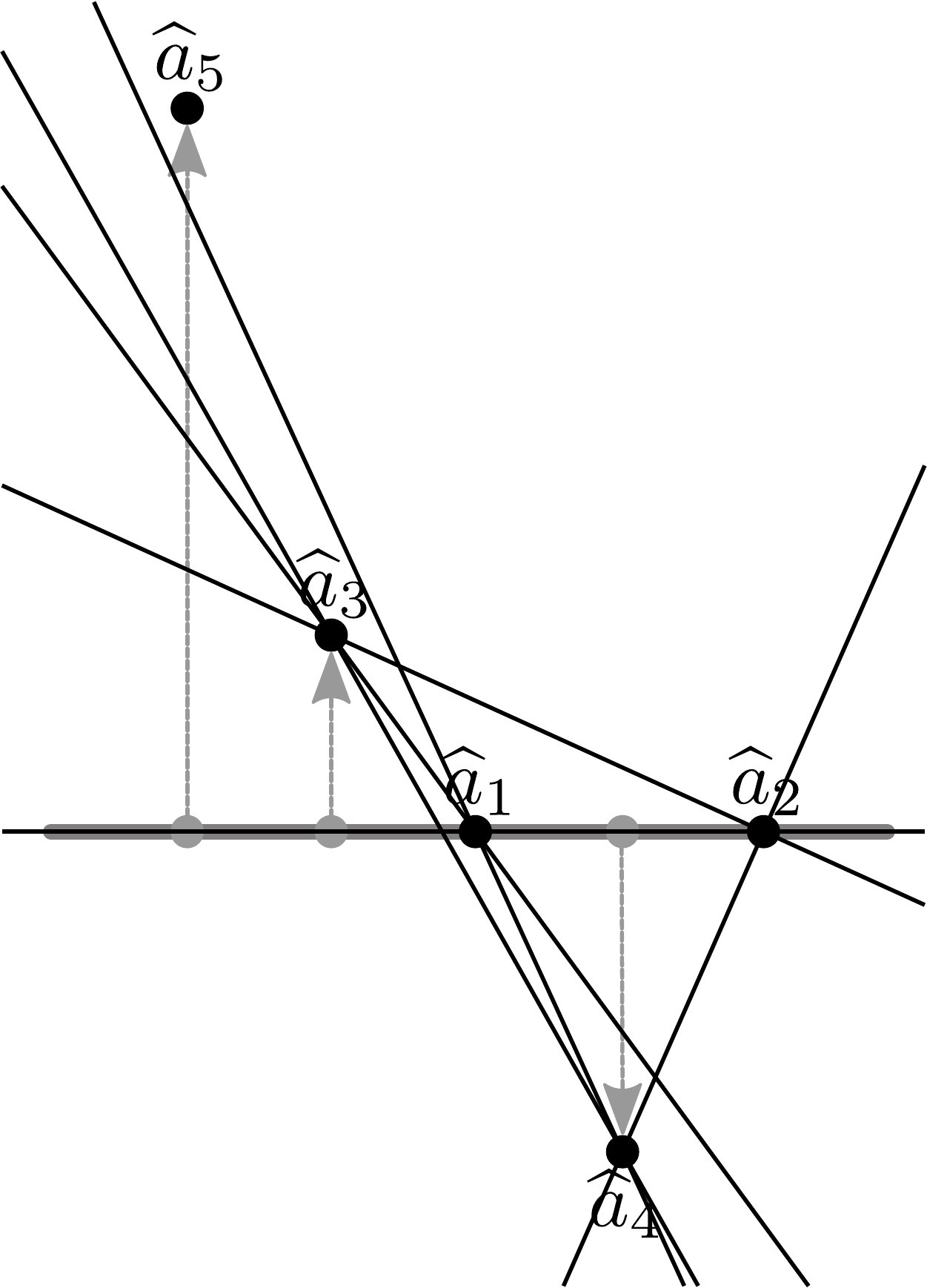}&
 \includegraphics[width=.22\linewidth]{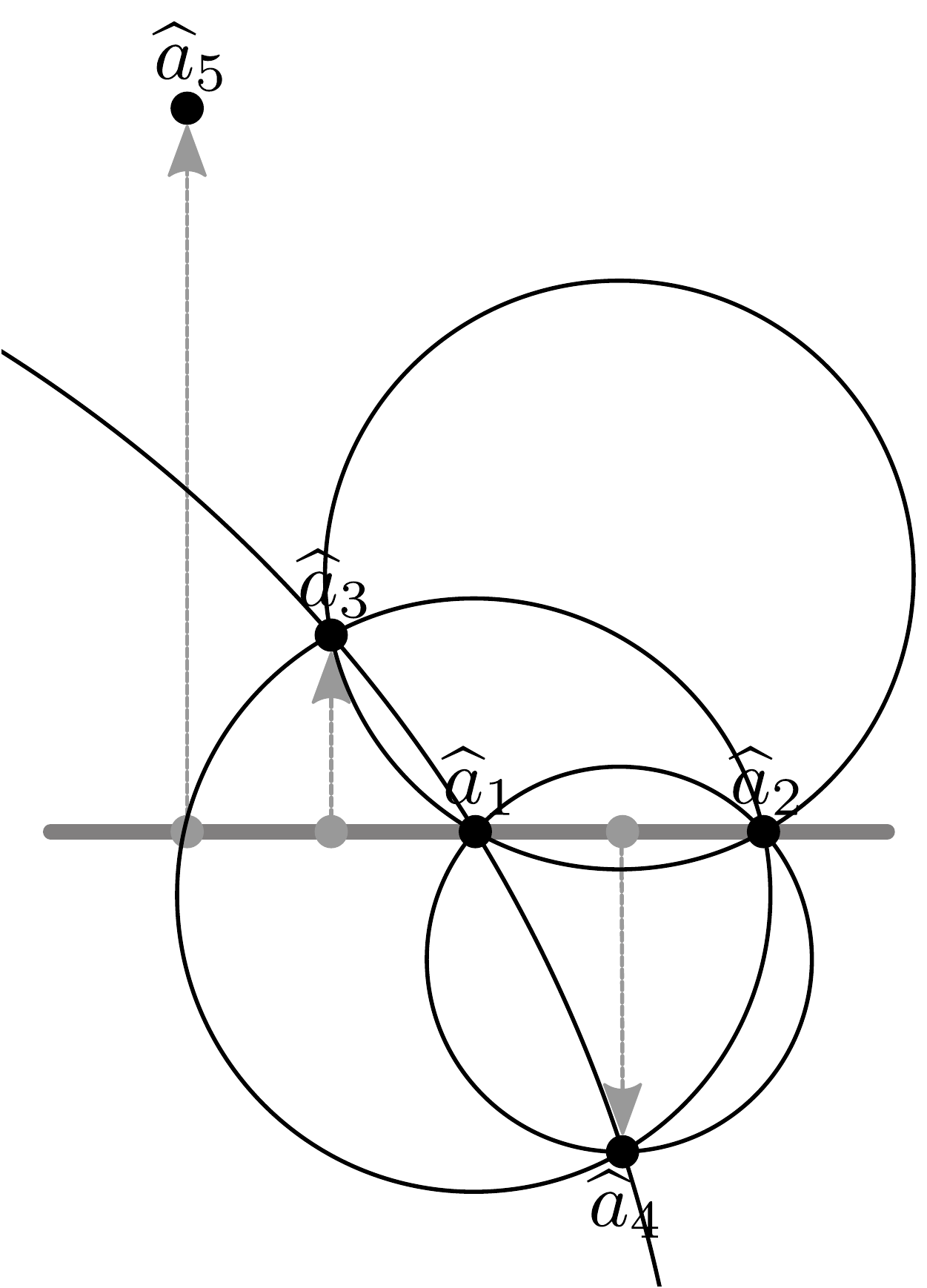}\\
 (a)&(b)&(c)&(d)
 \end{tabular}
 \caption{A point configuration $\{a_1,\dots,a_5\}$ (a) and one of its
$K$-liftings $\{\wh a_1,\dots,\wh a_5\}$ (b), where $K=\ball{d+2}$.
 It fulfills conditions~\ref{it:heights} (c) and~\ref{it:circumspheres} (d).
 }
\end{figure}

Our next steps require some extra notation. A face $F$ of a polytope $P$ is \defn{visible} from a
point $p$ if there is a point $x\in\relint(F)$ such that the segment
$[x, p]$ intersects $P$ only at $x$. In particular, if $p\in P$ is not a vertex, then no face is visible from $p$.

A facet of a $d$-polytope $P$ is a \defn{lower facet} if the last coordinate of
its outer normal vector is negative. The \defn{lower envelope} of $P$ is the
polytopal complex consisting of the lower facets of $P$ and their faces.
A face $F$ of $P$ is an \defn{equatorial face} if it admits a supporting 
hyperplane whose normal vector has $0$ as the last coordinate. 

\begin{lemma}\label{lem:visibleequator}
For any point configuration in general position $A$, and for any \lex-lifting
$\wh A$, every equatorial face of $\conv (\wh A\setminus \wh a_{{n}})$ is
visible from $\wh a_{{n}}$.
\end{lemma}

The following lemma is the link with the results presented in~\cite{Padrol2013}. 
The oriented matroid of a \lex-lifting of $A$ is completely determined by that of $A$ and
the signs of the heights. Indeed, it can be described using \defn{lexicographic extensions},
for which we refer to~\cite[Section 7.2]{OrientedMatroids1993} (see also~\cite[Section 4.1]{Padrol2013}).

\begin{lemma}\label{lem:om}
A \lex-lifting of $A$ with heights $h_i$ realizes the dual oriented matroid of a lexicographic extension of the Gale dual of~$A\setminus \{a_n\}$ with signature $[a_{{n-1}}^{-\signe(h_{{n}})\signe(h_{{n-1}})},\dots,a_{{d+2}}^{-\signe(h_{{n}})\signe(h_{{d+2}})}]$ .
\end{lemma}

This in particular implies the following result, whose proof we omit (see~\cite[Section 4.1]{Padrol2013}).
\begin{corollary}\label{cor:combinatorialtypelifting}
If $\wh A$ and $\wh A'$ are two \lex-liftings of $A$ with respective heights $h_ i$ and $h_i'$, and such that $h_ih_i'>0$ for all $1\leq i\leq n$, then $\wh A$ and $\wh A'$ have the same oriented matroid. In particular, $\conv(\wh A)$ and $\conv(\wh A')$ are combinatorially equivalent.
\end{corollary}

\subsection{Placing triangulations}

The combinatorics of the convex hulls of \lex-liftings are easily explained in terms of
\defn{lexicographic triangulations}. We refer to~\cite[Section
4.3]{DeLoeraRambauSantosBOOK} for a detailed presentation, and we will only
present here the parts that will be directly useful for us. Namely, \defn{placing
triangulations} and their relation to positive \lex-liftings.
Here, we say that a face of a triangulation of $A$ is \defn{visible} if it is 
contained in a visible face of $\conv(A)$.

\begin{lemma}[{\cite[Lemma~4.3.2]{DeLoeraRambauSantosBOOK}}]\label{lem:placing}
Let $A=\{a_1,\dots,a_n\}$ be a point configuration and let $\cT$ be a triangulation of the point configuration $A\setminus a_n$. 
Then there is a triangulation $\cT'$ of $A$ whose cells are 
\[\cT': = \cT\cup \set{\conv(B\cup a_n)}{B\text{ is a face of }\cT\text{ visible from }a_n}.\]
Moreover, $\cT'$ is the only triangulation of $A$ that contains all the cells of $\cT$.
\end{lemma}

\begin{definition}
The \defn{placing triangulation} of $A$ (with respect to the order induced by the labels) 
is the triangulation $\cT_n$ obtained iteratively as follows: $\cT_1$ is the trivial triangulation 
of $\{a_{{1}}\}$ and $\cT_i$ is the unique triangulation of $\{a_{{1}},\dots, a_{{i}}\}$ that contains $\cT_{i-1}$.
\end{definition}

\begin{lemma}[{\cite[Lemma~4.3.4]{DeLoeraRambauSantosBOOK}}]\label{lem:placinglifting}
Let $\wh A\subset \RR^{d+1}$ be a positive \lex-lifting of a configuration
$A\subset\RR^d$. Then the lower envelope of $\conv(\wh A)$
is combinatorially equivalent (as a simplicial complex) to the placing
triangulation of~$A$. 
\end{lemma}

\section{The construction}
Our construction is based on the following results, which show how certain
triangulations of certain $K$-liftings are always $K$-Delaunay
(Proposition~\ref{prop:delaunay}) and neighborly
(Proposition~\ref{prop:neighborly}). 

The proof of Proposition~\ref{prop:delaunay} is inspired by \cite[Proposition 1.3.1]{Gonskaphd} and \cite[Proposition 17]{GonskaZiegler2013}, where a similar argument is used to prove that the cyclic polytope is inscribable. 
\begin{proposition}\label{prop:delaunay}
Let $K$ be an \sscb-body in $\RR^{d+2}$, $\csp$ a $K$-stereographic
projection, $A=\{a_1,\dots,a_n\}$ a configuration of $n\geq d+2$
labeled points in general position in~$\RR^d$ 
 and  $\wh A$ a $K$-lifting of $A$. Then the $K$-Delaunay triangulation
$\cT:=\Del_K(\wh A)$ coincides with the placing triangulation of $\wh A$.
\end{proposition}
\begin{proof}
 The proof is by induction on $n$. If $n=d+2$ then both triangulations consist
of the single simplex spanned by $\wh A$. Otherwise, let $\cT'$ be the
$K$-Delaunay triangulation of $\wh A\setminus \wh a_{{n}}$. By induction
hypothesis, this is the placing triangulation of $\wh A\setminus \wh a_{{n}}$.
Moroever, since the $K$-lifting fulfills condition \ref{it:circumspheres},
every $K$-Delaunay cell of $\cT'$ is still a $K$-Delaunay cell of $\cT$.
Now, by Lemma~\ref{lem:placing} the placing triangulation is the
unique triangulation of $\wh A$ containing all the cells of $\cT'$.
\end{proof}

Proposition~\ref{prop:neighborly} can be deduced from the \emph{Gale sewing}
technique presented in \cite[Theorem 4.2]{Padrol2013}.
However, the original proof of the theorem
exploits Gale duality and oriented matroid theory, while this primal proof is
elementary. Moreover, this is the setting that will eventually allow us to prove
inscribability in Theorem~\ref{thm:inscribableandneighborly}.

\begin{proposition}\label{prop:neighborly}
Let $A=\{a_1,\dots,a_n\}\subset\RR^d$ be a configuration of $n\geq d+2$ labeled points 
in general position that is the set of vertices of a simplicial $d$-polytope~$P$, and let $\wh A$ be a \lex-lifting of~$A$. 
If $P$ is $k$-neighborly (as a polytope), then $\wh P:=\conv(\wh A)$ is $k$-neighborly (as a polytope) and the placing triangulation of $\wh A$ is $(k+1)$-neighborly (as a triangulation).
\end{proposition}
\begin{proof}
The proof of the first claim ($\wh P$ is $k$-neighborly) is straightforward. 
Indeed, every subset $S$ of $k$ points of $A$ forms a face of $P$. Let $H$ be a supporting hyperplane for this face and let $\wh H$ be its preimage under the projection. 
Then $\wh H$ is a hyperplane supporting an equatorial face of $\wh P$ that contains exactly the $k$ points of $\wh A$ 
corresponding to liftings of points in $S$. The claim follows because by construction every vertex of~$\wh P$ is the lifting of a distinct vertex of~$P$.

The second claim is proved by induction on $n$, and it is trivial when $n=d+2$. For $n>d+2$, let $\cT$ be the placing triangulation of $\wh A$ and $\cT'$ the corresponding placing triangulation of $\wh A\setminus \wh a_{{n}}$. Now fix a subset $S$ of $\wh A$ of size $k+1$. If $\wh a_{{n}}\notin S$, then $S$ forms a cell of $\cT'$ by induction hypothesis, and hence of $\cT$. Otherwise, if $\wh a_{{n}}\in S$, then $S'=S\setminus \wh a_{{n}}$ must be an equatorial face of $\wh A\setminus \wh a_{{n}}$. By Lemma~\ref{lem:visibleequator}, $S'$ is visible from $\wh a_{{n}}$ and hence by the definition of the placing triangulation, $S$ must be a cell of $\cT$.
\end{proof}

The combination of these two propositions directly proves our main result
(see also Section~\ref{sec:ubtb} for the relation with the Upper Bound Theorem 
for triangulations).
\begin{theorem}\label{thm:neighborlyanddelaunay}
Let $K$ be an \sscb-body in $\RR^{d+2}$, and $\csp$ a $K$-stereographic
projection. Let $A=\{a_1,\dots,a_n\}\subset \RR^d$ be a configuration of $n\geq d+2$
labeled points in general position that is the vertex set of a $k$-neighborly
$d$-polytope $P$, and let $\wh A$ be a $K$-lifting of $A$. Then $\wh P:=\conv(\wh A)$ is
a $k$-neighborly polytope with vertex set $\wh A$ and the $K$-Delaunay
triangulation of $\wh A$ is a $(k+1)$-neighborly triangulation. 
\end{theorem}

We can easily adapt Proposition~\ref{prop:neighborly} to obtain a statement in
terms of neighborly $K$-inscribable polytopes. To do so, we need the following
consequence of Lemma~\ref{lem:Kprojectionmatroid}, whose proof we omit since it
only needs a combinatorial description of the face lattice of a positive lifting
(cf.~\cite[Lemma~4.3.4]{DeLoeraRambauSantosBOOK}).
\begin{lemma}\label{lem:delaunayislifting}
 If the $K$-Delaunay triangulation of $A$ coincides with its placing
triangulation, then the polytope $P_A$ of Lemma~\ref{lem:KdelaunayKinscribed} is
combinatorially equivalent to the convex hull of a positive \lex-lifting of
$A\cup\{a_{n+1}\}$ for any $a_{n+1}\in\RR^d$.
\end{lemma}

As a direct consequence of the combination of
Lemma~\ref{lem:KdelaunayKinscribed} and Theorem~\ref{thm:neighborlyanddelaunay},
we obtain the following result. Observe how the strategy is to start with a
$k$-neighborly $d$-polytope, lift it to a $(k+1)$-neighborly
$(d+1)$-triangulation, and lift it again (with a positive lifting with the same
order) to a $(k+1)$-neighborly $(d+2)$-polytope, which is inscribable.

\begin{theorem}\label{thm:inscribableandneighborly}

Let $A=\{a_1,\dots,a_n\}\subset \RR^d$ be a 
 configuration of $n\geq d+3$ points in general position
that is the vertex set of a $k$-neighborly $d$-polytope~$P$, let $\wh A$ be a \lex-lifting of $A$ and
let $\wwh{A}$ be a positive \lex-lifting of~$\wh A$. Then:
\begin{enumerate}[(i)]
 \item $\wwh P:=\conv(\wwh A)$ is a $(k+1)$-neighborly $(d+2)$-polytope.
In particular, if $P$ is neighborly,
so is $\wwh P$.
\item for each \sscb-body $K$ in $\RR^{d+2}$, there is a $K$-inscribed realization of (a polytope combinatorially equivalent to) $\wwh P$.
\end{enumerate}
That is, $\wwh P$ is a $K$-inscribable
$(k+1)$-neighborly $(d+2)$-polytope. 
\end{theorem}
\begin{proof}
Recall that the oriented matroid of $\wh A$ and $\wwh A$, as well as the combinatorial type
of $\wh P$ and $\wwh P$, only depends on the signs of the heights of the \lex-liftings, by Corollary~\ref{cor:combinatorialtypelifting}.
Hence, we can take $\wh A$ to be a \lex-lifting that is also a $K$-lifting. By
Proposition~\ref{prop:delaunay}, the placing triangulation of $\wh A\setminus \wh a_n$
coincides with its $K$-Delaunay triangulation. Therefore, by
Lemma~\ref{lem:delaunayislifting}, $\wwh P$ is combinatorially equivalent
to the $K$-inscribed polytope $P_{\wh A\setminus \wh a_n}$ of
Lemma~\ref{lem:KdelaunayKinscribed}. This proves inscribability.

 For $k$-neighborliness, observe that as a consequence of  Lemma~\ref{lem:placinglifting}, 
 every face of the placing triangulation of  $\wh A$ is also a face of $\wwh P$. 
 Since the triangulation is $(k+1)$-neighborly by Proposition~\ref{prop:neighborly}, then so is $\wwh P$.
\end{proof}

Theorem~\ref{thm:inscribableandneighborly} provides the following method to
construct many $K$-inscribable neighborly polytopes with $n$
vertices starting with an arbitrary polygon or $3$-polytope.

\begin{construction}\label{cons:construction}
For a fixed \sscb-body $K\subset \RR^d$, to construct a $K$-inscribable
neighborly $d$-polytope with $n$ vertices $P$ and a $(d-1)$-dimensional set of
$n-1$ points with a neighborly $K$-Delaunay triangulation $\cT$:
\begin{enumerate}[label*=\arabic*.,ref={\arabic*}]
 \item Set $n_0:=n-2\ffloor{d-2}{2}$ and $d_0:=d-2\ffloor{d-2}{2}$
 \item Let $P_0$ be any simplicial $d_0$-polytope with $n_0$ vertices.
 \item $P_0$ is neighborly because $d_0\in\{2,3\}$.
 \item Set $A_0=\verts(P_0)$.
 \item\label{it:loop} For $i$ from $1$ to $\ffloor{d-2}{2}$ do:
 \begin{enumerate}[label*=\arabic*.,ref={\arabic*}]
      \item\label{it:reorder} Choose a permutation $\sigma\in\cS_{n_{i-1}}$ and relabel the points of $A_{i-1}$ with $a_j\mapsto a_{\sigma(j)}$.
      \item\label{it:addpoints1} Let $B_{i-1}:=A_{i-1}\cup \{a_{n_{i-1}+1}\}$ for some $a_{n_{i-1}+1}\in \RR^{d_{i-1}}$.
      \item\label{it:lifting1} Let $\wh B_{i-1}$ be a \lex-lifting of $B_{i-1}$.
      \item \label{it:addpoints2} Let $\wh C_{i-1}:=\wh B_{i-1}\cup \{\wh a_{n_{i-1}+2}\}$ for some $\wh a_{n_{i-1}+2}\in \RR^{d_{i-1}+1}$.
      \item\label{it:lifting2} Set $A_{i}:=\wwh {C}_{i-1}$ to be a positive
\lex-lifting of $\wh C_{i-1}$. 
      \item Set $n_{i}:=n_{i-1}+2$ and $d_{i}:=d_{i-1}+2$.
      \item By Theorem~\ref{thm:inscribableandneighborly}, $P_i:=\conv(A_i)$ is a neighborly $d_i$-polytope with $n_i$ vertices.
    \end{enumerate}
\item $P:=P_{\ffloor{d-2}{2}}$ is a $K$-inscribable neighborly $d$-polytope with
$n$ vertices by Theorem~\ref{thm:inscribableandneighborly}.
\item $\cT:=\Del_K(\wh A_{\ffloor{d-4}{2}})$ is a neighborly $K$-Delaunay
triangulation of a $(d-1)$-dimensional set of $n-1$ points by
Theorem~\ref{thm:neighborlyanddelaunay}.

\end{enumerate}
\end{construction}

\begin{remarks}\label{rmks:construction}
Some observations concerning Construction~\ref{cons:construction}:
\begin{enumerate}[a)]
 \item Observe that we start the construction with $n-2\ffloor{d-2}{2}$ points
instead of $n$ points. The missing points are added at steps \ref{it:loop}.\ref{it:addpoints1}
and \ref{it:loop}.\ref{it:addpoints2}. This is due to the fact that the
combinatorics of the convex hull and triangulations of a \lex-lifting are
independent of the position of the last point.
 
 \item
 We deferred the discussion about the order of the points for the \lex-lifting
until now. However, the relabeling step in Construction~\ref{cons:construction}
is crucial, since it is the choice of these permutations that produces the
variety of combinatorial types. It is also important to remark that the choice
of the permutation is only done once every two dimensions, since the liftings of steps \ref{it:loop}.\ref{it:lifting1} 
and \ref{it:loop}.\ref{it:lifting2} need to follow the same
order or otherwise Theorem~\ref{thm:inscribableandneighborly} does not hold.

\item 
If all the liftings are positive, and no relabeling is done, then this construction
yields cyclic polytopes (cf. \cite[Proposition~4.7]{Padrol2013}).
\item%
 Construction~\ref{cons:construction} relies on \lex-liftings,
which in their definition depend on some $h_i$'s being ``large enough''. One
might wonder how feasible, computationally, it is to find these $h_i's$ (at least
when $K=\ball{d}$). On the one
hand, in the case of \lex-liftings or $\ball{d}$-liftings, it is not hard to
find an upper bound for the minimal valid $h_i$ that depends only on $d$ and $n$
if we assume, for example, that the $a_i$'s have integer entries whose absolute
value is bounded by some number $M$. Indeed, the conditions required in
Definition~\ref{def:lexicographiclifting} only depend on certain determinants
being positive. However, if we apply these bounds to construct our point
configurations from scratch, we will end up with points having extremely
 large coordinates. It remains open which are the minimal coordinates needed to realize these configurations. Can they be  realized with polynomially large coordinates? In lower dimensions this kind of questions has already been considered (see~\cite{Erickson2002} and~\cite{Erickson2003}). 

 \end{enumerate}
\end{remarks}

\section{Lower bounds} 

It remains to discuss how many different labeled combinatorial types of
inscribable neighborly polytopes (and of neighborly $K$-Delaunay triangulations)
can be obtained with Construction~\ref{cons:construction}. These bounds were
obtained in~\cite{Padrol2013} using a construction that can be seen to be
equivalent. We will only sketch the main ideas for the originial proof, which is
based on oriented matroids, and refer to~\cite{Padrol2013} and~\cite[Chapter
5]{Padrolphd} for details.

 The main ingredient for these bounds is the following lemma from \cite{Padrol2013}, for which we sketch its proof. It ignores the variability provided by the signs of the $h_i$'s and only focuses on positive liftings. Stronger bounds are discussed in~\cite[Chapter 5]{Padrolphd}.%

\begin{lemma}[{\cite[Proposition~6.1]{Padrol2013}}, {\cite[Lemma~6.6]{Padrol2013}}]
Let $A$ be a configuration of $n> d+2$ labeled points in general position in~$\RR^d$, where $d\geq 2$ is even, that is the vertex set of a neibhborly polytope. 

Denote by $\cP$ the set of different labeled combinatorial types of polytopes~$P$ that fulfill:
\begin{enumerate}[(i),leftmargin=*]
 \item $P=\conv(\wwh B)$ is obtained by doing a couple of positive \lex-liftings
(as in Theorem~\ref{thm:inscribableandneighborly}) of some configuration $B$
obtained by relabeling the elements of $A$ and adding two extra points, and
 \item the labeled combinatorial type of $\conv(A)$ can be recovered from that of $P$.
\end{enumerate}

Then there are at least $\frac{(n+1)!}{(d+2)!}$ different elements in $\cP$.
\end{lemma}
\begin{proof}[Proof idea]
To obtain $P$, first a permutation $\sigma$ is applied to the labels of $A$.
Next, two elements $\{a_{n+1},a_{n+2}\}$ are appended to obtain $B$. Then $\wh
B$ is a positive \lex-lifting of $B$ and $\wwh{B}$ is a positive \lex-lifting of
$\wh B$ (with respect to the same ordering). Finally the permutation
$\sigma^{-1}$ is applied on the labels of $\wwh B$ so that the original labeling
of $A$ is preserved.

In {\cite[Proposition~6.1]{Padrol2013}} it is shown that, if we restrict to the
case when both \lex-liftings are positive, then the $n-d-2$ first elements of
$\sigma$ can be recovered from the combinatorics of~$P$, and that $2$ different
cases can be distinguished for the last $d+2$ elements. The proof relies on a result of 
Shemer \cite[Theorem 2.12]{Shemer1982} that states that even dimensional neighborly polytopes are rigid
in the oriented matroid sense. That is, that the combinatorial type of an even dimensional neighborly 
polytope completely determines the combinatorial type of all its subpolytopes (convex hulls of subsets
of vertices).

A final step is that, when $\conv(A)$ is neighborly, if $a_{n+2}$ and the face lattice of $P$ are fixed, then there are at most two points of $\wwh B\setminus a_{n+2}$ that could be $a_{n+1}$. Hence the label of $a_{n+1}$ can also be chosen almost arbitrarily, adding a factor of $\frac{n+1}{2}$ to the number of combinatorial types.
\end{proof}

Applying recursively this lemma one obtains the following bound, of order $n^{\frac{d}{2}n(1+o(1))}$ as $n\rightarrow \infty$, which is estimated using Euler-Maclaurin approximation.
\begin{theorem}\label{thm:bound}
 For every \sscb-body $K\subset \RR^{d+1}$, where $d$ is even, the number $\inpK{n}{d}$ of different labeled combinatorial types of $K$-inscribable neighborly $d$-polytopes with $n$ vertices, fulfills 
\begin{align}
  \inpK{n}{d}&\geq\prod_{i=1}^{\frac{d}{2}} {\frac{(n-d-1+2i)!}{(2i)!}}\notag\\
&\geq\frac{\left(n-1\right) ^{\left( \frac{n-1}{2}\right) ^{2}}}{{(n-d-1)}^{{\left(\frac{n-d-1}{2}\right)}^{2}} d  ^{{(\frac{d}{2})}^{2}}{\e^{\frac{3d(n-d-1)}{4}}}}\geq\left( \frac{n-1}{\e^{3/2}}\right)^{\frac12d(n-d-1)}.\notag
\end{align}
\end{theorem}

By Lemma~\ref{lem:KdelaunayKinscribed}, the same bound holds for the number of labeled combinatorial types of neighborly Delaunay triangulations of configurations of $n-1$ points in $\RR^{d-1}$.

 For the odd-dimensional case, observe that if $H$ is a hyperplane that intersects $K$ in the interior, then $K\cap H$ is also an \sscb-body and the pyramid over any even dimensional $(K\cap H)$-inscribable neighborly polytope is $K$-inscribable and neighborly. This shows that the bound $n^{\frac{d}{2}n(1+o(1))}$ also applies for odd dimensional configurations.
 Moreover, although the use of labeled types is needed to prove these bounds, observe that we can easily recover bounds for non-labeled combinatorial types just by dividing by $n!$. The bounds obtained this way are still superexponential and of comparable asymptotics.

\section{Final observations}\label{sec:observations}
\subsection{Relaxing the conditions}
Our statements concern smooth strictly convex bodies. However, our proofs only use a single smooth strictly convex point in $\partial K$; 
namely, the center $\np$ of the $K$-stereographic projection $\csp$. Hence, our results can be generalized to bodies that have one smooth strictly convex point in their boundary.

\subsubsection{Smoothness}
Actually, all our results also hold directly if we completely remove the hypothesis of smoothness (while keeping strict convexity). 
Indeed, for any $d$-dimensional convex body, the set of singular (non-smooth) points of $\partial K$ always has $(d-1)$-dimensional Hausdorff measure zero (cf. \cite[Theorem~2.2.4]{Schneider}), so there is always a smooth point in $\partial K$.

\subsubsection{Strict convexity}
In contrast, strict convexity cannot be avoided. For example, no simplicial polytope with more than $d(d+1)$ vertices can be inscribed on the boundary of a $d$-simplex, since otherwise there would be $d+1$ vertices lying in a common facet, which cannot happen for simplicial polytopes.

\subsection{Delaunay triangulations}

\subsubsection{Upper bound theorem for balls}\label{sec:ubtb}
A consequence of Stanley's Upper Bound Theorem for simplicial spheres \cite{StanleyCCA} 
is the following characterization of triangulations with maximal number of faces.

\begin{theorem}[{Upper Bound Theorem for simplicial balls
\cite[Corollary~2.6.5]{DeLoeraRambauSantosBOOK}}]\label{thm:ubtb}
A triangulation has the maximal number of
$i$-dimensional cells among all triangulations of sets of $n$ points in $\RR^d$
if and only if it is neighborly and its convex hull is a $d$-simplex.
\end{theorem}

However, the convex hull of the neighborly $K$-Delaunay
triangulations obtained with Theorem~\ref{thm:neighborlyanddelaunay} 
 is not a simplex. To obtain
such triangulations, we need neighborly polytopes with a $K$-inscribed
realization that can be stacked on a facet such that the result is still
$K$-inscribed (see \cite{Gonskaphd} or \cite{Seidel1991}).
 
Construction~\ref{cons:construction} can be modified to get them, following
\cite[Remark 1.3.5]{Gonskaphd}. Indeed, at the beginning of the last iteration of the construction (from
$B_{\ffloor{d-4}{2}}$ to $\wh B_{\ffloor{d-4}{2}}$), apply a stellar
subdivision to the the first simplex directly after it appears, and then continue as usual. What we get
then is the vertex projection of a $K$-inscribed neighborly polytope that has been
stacked with a point $p$ on $\partial K$. We can then apply a
$K$-stereographic projection with center $p$ to get the desired triangulation
with a maximal number of faces.

\subsubsection{Other Delaunay triangulations} 
The concept of Delaunay triangulation has several generalizations. In particular, 
for non-Euclidean metrics. 
Construction~\ref{cons:construction} can be
suitably adapted for many of these generalizations, 
providing many neighborly triangulations. For example, this can be done
for all smooth strictly convex distance functions, which are those
whose unit ball is an \sscb-body.
These should not be confused with our $K$-Delaunay triangulations. 
Although Delaunay triangulations for the Euclidean metric coincide with $\ball{d+1}$-Delaunay 
triangulations, this is a rare phenomenon: $K$-Delaunay triangulations are always 
regular, which does not happen for those given by most \sscb-distances~\cite{Santos1996}.

\subsection{Universally inscribable polytopes}\label{sec:universallyinscribable}

We showed that Gale sewn polytopes are inscribable in any \sscb-body. We will 
call a polytope \defn{universally inscribable} if it has this property. 

\subsubsection{Lawrence polytopes}

The following observation is due to Karim Adiprasito (personal communication): 
Lawrence polytopes are always universally inscribable. Indeed, the arguments used in
\cite[Proposition~6.5.8]{AdiprasitoZiegler2014} to prove their inscribability
extend naturally to any \sscb-body~$K$.

\subsubsection{Stacked polytopes}
Since we have seen that there are universally inscribable neighborly polytopes,
it is natural to ask whether there are also universally inscribable stacked polytopes with an arbitrary number of vertices.
The answer is yes. Take a sequence of  positive numbers 
$a_1\ll a_2\ll \cdots \ll a_n$ and consider the point configuration 
\[A=\{e_1,\dots, e_{d-1},-\sum_{i=1}^{d-1} e_i,a_1 e_d, a_2e_d,\dots,a_n e_d\},\]
where $e_i$ are the standard basis vectors. Then the $K$-Delaunay subdivision
of $A$ is an iterated stellar subdivision of the same face of a simplex, which 
is lifted to a stacked $K$-inscribed polytope whose stacked triangulation's dual graph is a path.

\subsubsection{Inscribability does not imply universal inscribability}
Another observation due to Karim Adiprasito (personal communication) is that 
there exist polytopes that are inscribable in the ball but that are not universally inscribable. 
The key 
is the infinite family of projectively unique inscribed polytopes
in a fixed dimension constructed by himself together with G\"unter Ziegler \cite[Theorem~6.5.7]{AdiprasitoZiegler2014}.

Consider an analytic convex body~$K$ in which all of the Adiprasito--Ziegler projectively unique polytopes can be inscribed. Then one of the $4$-dimensional sections of $K$ must be a quadric. Indeed, every Adiprasito--Ziegler polytope has a $4$-dimensional section that is a \emph{cross-bedding cubical torus} (see \cite{AdiprasitoZiegler2014}). 
Since the vertices of this section have an accumulation point by compactness, the analyticity of the section implies that the section with~$K$ is itself a quadric. (Recall that each of these polytopes is inscribable in the ball and projectively unique, and hence in any of their realizations the cross-bedding cubical tori are inscribed in a quadric.)

Since it is easy to construct analytic convex bodies without such a section, the claim follows.

\section*{Acknowledgements}
We would like to thank Karim Adiprasito for sharing his insights with us and letting us add them to our paper; and Moritz Firsching
for checking inscribability of neighborly $4$-polytopes with few vertices.

\bibliographystyle{plain}
\bibliography{InscribedNeighborly}

\begin{thebibliography}{10}

\bibitem{AdiprasitoPadrolTheran2014}
Karim~A. Adiprasito, Arnau Padrol, and Louis Theran.
\newblock {Universality theorems for inscribed polytopes and Delaunay
  triangulations}.
\newblock {\em {Discrete Comput. Geom.}}, in press.
\newblock Preprint available at \href{http://arxiv.org/abs/1406.7831}{
  arXiv:1406.7831}.

\bibitem{AdiprasitoZiegler2014}
Karim~A. Adiprasito and G{\"u}nter~M. Ziegler.
\newblock {Many projectively unique polytopes}.
\newblock {\em Invent. Math.}, pages 1--72, 2014.

\bibitem{Alon1986}
Noga Alon.
\newblock {The number of polytopes, configurations and real matroids}.
\newblock {\em Mathematika}, 33(1):62--71, 1986.

\bibitem{BarvinokNovik2008}
Alexander Barvinok and Isabella Novik.
\newblock {A centrally symmetric version of the cyclic polytope}.
\newblock {\em Discrete Comput. Geom.}, 39(1-3):76--99, 2008.

\bibitem{OrientedMatroids1993}
Anders Bj{\"o}rner, Michel {Las Vergnas}, Bernd Sturmfels, Neil White, and
  G{\"u}nter~M. Ziegler.
\newblock {\em {Oriented matroids.}}
\newblock {Encyclopedia of Mathematics and Its Applications. 46. Cambridge:
  Cambridge University Press. 516 p. }, 1993.

\bibitem{Brown1979}
Kevin~Q. Brown.
\newblock {Voronoi diagrams from convex hulls.}
\newblock {\em Inf. Process. Lett.}, 9:223--228, 1979.

\bibitem{Caratheodory1911}
Constantin Carath{\'e}odory.
\newblock {{\"U}ber den Variabilit{\"a}tsbereich der {Fourier}'schen Konstanten
  von positiven harmonischen Funktionen.}
\newblock {\em Rendiconto del Circolo Matematico di Palermo}, 32:193--217,
  1911.

\bibitem{DeLoeraRambauSantosBOOK}
Jes{\'u}s~A. {De Loera}, J{\"o}rg Rambau, and Francisco Santos.
\newblock {\em {Triangulations}}, volume~25 of {\em {Algorithms and Computation
  in Mathematics}}.
\newblock Springer-Verlag, Berlin, 2010.
\newblock Structures for algorithms and applications.

\bibitem{Edelsbrunner2006}
Herbert Edelsbrunner.
\newblock {\em {Geometry and topology for mesh generation. }}.
\newblock {Cambridge: Cambridge University Press}, 2006.

\bibitem{Erickson2002}
Jeff Erickson.
\newblock {Dense point sets have sparse Delaunay triangulations or ``\dots but
  not too nasty''.}
\newblock In {\em {Proceedings of the thirteenth annual ACM-SIAM symposium on
  Discrete algorithms}}, pages 125--134. {Philadelphia, PA: Society for
  Industrial and Applied Mathematics (SIAM)}, 2002.

\bibitem{Erickson2003}
Jeff Erickson.
\newblock {Nice point sets can have nasty Delaunay triangulations.}
\newblock {\em Discrete Comput. Geom.}, 30(1):109--132, 2003.

\bibitem{Gonskaphd}
Bernd Gonska.
\newblock {\em {Inscribable polytopes via Delaunay triangulations}}.
\newblock {P}h.{D}. thesis, Freie Universit{{\"a}t} Berlin, 2013.

\bibitem{GonskaZiegler2013}
Bernd {Gonska} and G{\"u}nter~M. {Ziegler}.
\newblock {Inscribable stacked polytopes.}
\newblock {\em {Adv. Geom.}}, 13(4):723--740, 2013.

\bibitem{GoodmanPollack1986}
Jacob~E. Goodman and Richard Pollack.
\newblock {Upper bounds for configurations and polytopes in {${\bf R}^d$}.}
\newblock {\em Discrete Comput. Geom.}, 1:219--227, 1986.

\bibitem{GrunbaumConvexPolytopes}
Branko Gr{\"u}nbaum.
\newblock {\em {Convex polytopes}}, volume 221 of {\em {Graduate Texts in
  Mathematics}}.
\newblock Springer-Verlag, New York, second edition, 2003.
\newblock Prepared and with a preface by Volker Kaibel, Victor Klee and
  G{\"u}nter M. Ziegler.

\bibitem{HodgsonRivinSmith1992}
Craig~D. Hodgson, Igor Rivin, and Warren~D. Smith.
\newblock {A characterization of convex hyperbolic polyhedra and of convex
  polyhedra inscribed in the sphere}.
\newblock {\em Bull. Amer. Math. Soc. (N.S.)}, 27(2):246--251, 1992.

\bibitem{mathoverflow}
Sergei~Ivanov (http://mathoverflow.net/users/4354/sergei ivanov).
\newblock {Can all convex polytopes be realized with vertices on surface of
  convex body?}
\newblock MathOverflow.
\newblock URL:http://mathoverflow.net/q/107113 (version: 2012-09-13).

\bibitem{McMullen1970}
Peter McMullen.
\newblock {The maximum numbers of faces of a convex polytope.}
\newblock {\em Mathematika, Lond.}, 17:179--184, 1970.

\bibitem{Padrol2013}
Arnau Padrol.
\newblock {Many neighborly polytopes and oriented matroids.}
\newblock {\em {Discrete Comput. Geom.}}, 50(4):865--902, 2013.

\bibitem{Padrolphd}
Arnau Padrol.
\newblock {\em {Neighborly and almost neighborly configurations, and their
  duals}}.
\newblock {P}h.{D}. {T}hesis, {Universitat Polit{\`e}cnica de Catalunya}, 2013.

\bibitem{PadrolTheran2014}
Arnau Padrol and Louis Theran.
\newblock {Delaunay Triangulations with Disconnected Realization Spaces}.
\newblock In {\em {Proceedings of the Thirtieth Annual Symposium on
  Computational Geometry}}, {SOCG'14}, pages 163--170, New York, NY, USA, 2014.
  ACM.

\bibitem{Rivin1996}
Igor Rivin.
\newblock {A characterization of ideal polyhedra in hyperbolic {$3$}-space}.
\newblock {\em Ann. of Math. (2)}, 143(1):51--70, 1996.

\bibitem{Santos1996}
F.~{Santos}.
\newblock {On Delaunay oriented matroids for convex distance functions.}
\newblock {\em {Discrete Comput. Geom.}}, 16(2):197--210, 1996.

\bibitem{Schneider}
Rolf Schneider.
\newblock {\em {Convex bodies: the {B}runn-{M}inkowski theory}}, volume~44 of
  {\em {Encyclopedia of Mathematics and its Applications}}.
\newblock Cambridge University Press, Cambridge, 1993.

\bibitem{Schramm1992}
Oded Schramm.
\newblock {How to cage an egg}.
\newblock {\em Invent. Math.}, 107(3):543--560, 1992.

\bibitem{Seidel1985}
Raimund Seidel.
\newblock {A method for proving lower bounds for certain geometric problems}.
\newblock In G.~T. Toussaint, editor, {\em {Computational Geometry}}, pages
  319--334. North-Holland, Amsterdam, Netherlands, 1985.

\bibitem{Seidel1987}
Raimund Seidel.
\newblock {On the number of faces in higher-dimensional Voronoi diagrams}.
\newblock In {\em {Proceedings of the third annual symposium on Computational
  geometry}}, {SCG '87}, pages 181--185, New York, NY, USA, 1987. ACM.

\bibitem{Seidel1991}
Raimund Seidel.
\newblock {Exact upper bounds for the number of faces in $d$-dimensional
  Voronoi diagrams.}
\newblock {\em Applied Geometry and Discrete Mathematics: The Victor Klee
  Festschrift}, 4:517--530, 1991.

\bibitem{Shemer1982}
Ido Shemer.
\newblock {Neighborly polytopes.}
\newblock {\em Isr. J. Math.}, 43:291--314, 1982.

\bibitem{StanleyCCA}
Richard~P. Stanley.
\newblock {\em {Combinatorics and commutative algebra}}, volume~41 of {\em
  {Progress in Mathematics}}.
\newblock Birkh{\"a}user Boston Inc., Boston, MA, second edition, 1996.

\bibitem{Steiner1832}
Jacob Steiner.
\newblock {\em {\em {S}ystematische {E}ntwicklung der {A}bh{\"a}ngigkeit
  geometrischer {G}estalten von einander}}.
\newblock {Fincke, Berlin}, 1832.
\newblock Also in: Gesammelte Werke, Vol.~1, Reimer, Berlin 1881, pp. 229--458.

\bibitem{Steinitz1928}
Ernst Steinitz.
\newblock {{\"U}ber isoperimetrische {P}robleme bei konvexen {P}olyedern}.
\newblock {\em J. Reine Angew. Math.}, 159:133--143, 1928.

\bibitem{Sturmfels1987}
Bernd Sturmfels.
\newblock {Cyclic polytopes and d-order curves.}
\newblock {\em Geom. Dedicata}, 24:103--107, 1987.

\end{thebibliography}
\end{document}